\documentclass[absolute]{mymathart}
\usepackage{mymathmacros}
\usepackage{color,subfigure}

 \newtheoremstyle{numberedstyle}
   {9pt}
   {9pt}
   {\normalfont}
   {}
   {\bfseries}
   {.}
   {\newline}
   {}

\numberwithin{equation}{section}

\newtheorem{thm}{Theorem}[section]%
\newtheorem{lem}[thm]{Lemma}%
\newtheorem{cor}[thm]{Corollary}%
\newtheorem{prop}[thm]{Proposition}%

\newcommand{\W}{\mathcal{W}}
\newcommand{\bdd}{\operatorname{bdd}}

\theoremstyle{numberedstyle}
\newtheorem{question}[thm]{Question}%
\newtheorem{defn}[thm]{Definition}%

\newcommand{\B}{\mathcal{B}}
\renewcommand{\H}{\mathbb{H}}
\newcommand{\e}{\operatorname{e}}

\newcommand{\real}{\operatorname{real}}

\usepackage{hyperref}

\title[Absence of wandering domains]{Absence of wandering domains \\
        for some real entire functions \\ 
       with bounded singular sets}
\author{Helena Mihaljevi\'{c}-Brandt}
\address{Mathematisches Seminar, Christian-Albrechts-Universit\"{a}t zu Kiel, 
24118 Kiel, Germany}
\email{helena.m-b@t-online.de}
\author{Lasse Rempe-Gillen}
\address{Dept.\ of Mathematical Sciences, University of Liverpool, Liverpool L69 7ZL, UK}
\email{l.rempe@liverpool.ac.uk}
\thanks{Both authors were supported by EPSRC 
EP/E017886/1. The second author was supported by EPSRC Fellowship EP/E052851/1.
Both authors gratefully acknowledge support received through the European CODY network.}

\renewcommand{\S}{\mathcal{S}}
\newcommand{\len}{\operatorname{len}}
\newcommand{\F}{\mathcal{F}}
\newcommand{\J}{\mathcal{J}}
\newcommand{\Vt}{\widetilde{V}}

\begin{document} 
 \begin{abstract}
  Let $f$ be a real entire function whose set $S(f)$ of
   singular values is real and bounded. We show that, if
   $f$ satisfies a certain function-theoretic condition
   (the ``sector condition''), then $f$ has no wandering domains.
   Our result includes all maps
   of the form $z\mapsto \lambda \frac{\sinh(z)}{z} + a$ with
   $\lambda>0$ and $a\in\R$. 

  We also show the absence of wandering domains
   for certain non-real entire functions for which $S(f)$ is bounded
   and $f^n|_{S(f)}\to\infty$ uniformly. 

  As a special case of our theorem, we give a short, elementary
   and non-technical proof that the Julia set of the exponential map
   $f(z)=e^z$ is the entire complex plane. 

  Furthermore, we apply similar methods to extend a result of Bergweiler, concerning
   Baker domains of entire functions and their relation to the postsingular set, to
   the case of meromorphic functions.
 \end{abstract}
 
 \maketitle

 \section{Introduction}

 Let $f:\C\to\C$ be an entire function. The \emph{Fatou set}
 $\F(f)$ consists of those points near which the family $(f^n)$ of iterates
  of $f$ is equicontinuous with respect to the spherical metric. That is,
  $\F(f)$ is the set where the dynamics of $f$ is regular, while the dynamics
  on its complement
  $\J(f) := \C\setminus \F(f)$ (the \emph{Julia set}) is chaotic.   
  A \emph{wandering domain} of $f$ is a connected component
  $U$ of $\F(f)$ such that $f^n(U)\cap f^m(U)=\emptyset$ whenever $n\neq m$.

 Problems concerning wandering domains tend to be difficult. 
  The question whether polynomials (and rational maps)
  can have wandering domains remained open for the best part of the
  twentieth century, until 
  Sullivan gave a negative answer in his celebrated
  \emph{No Wandering Domains Theorem} \cite{sullivan},
  using quasiconformal deformation theory. 
  This proof generalises to the class $\S$ of transcendental
  entire functions $f$
  for which the set $S(f)$ of \emph{singular values} is finite 
\cite[Theorem 3]{eremenko_lyubich_2}.
  (See the end of this section for definitions.) 
  Functions with infinitely many singular values
  may well have wandering domains; elementary examples are given by
  $f(z)=z-1 + e^{-z} + 2\pi i$, 
  $f(z)=z+2\pi \sin(z)$ or
   $f(z) = z + \lambda \sin(2\pi z) + 1$ for suitable $\lambda$;
   see \cite[p. 168]{waltermero}. 

\begin{figure}
\begin{center}
 \subfigure[$f_1(z)=z+2\pi\sin(z)$]{\includegraphics[width=\textwidth]{%
    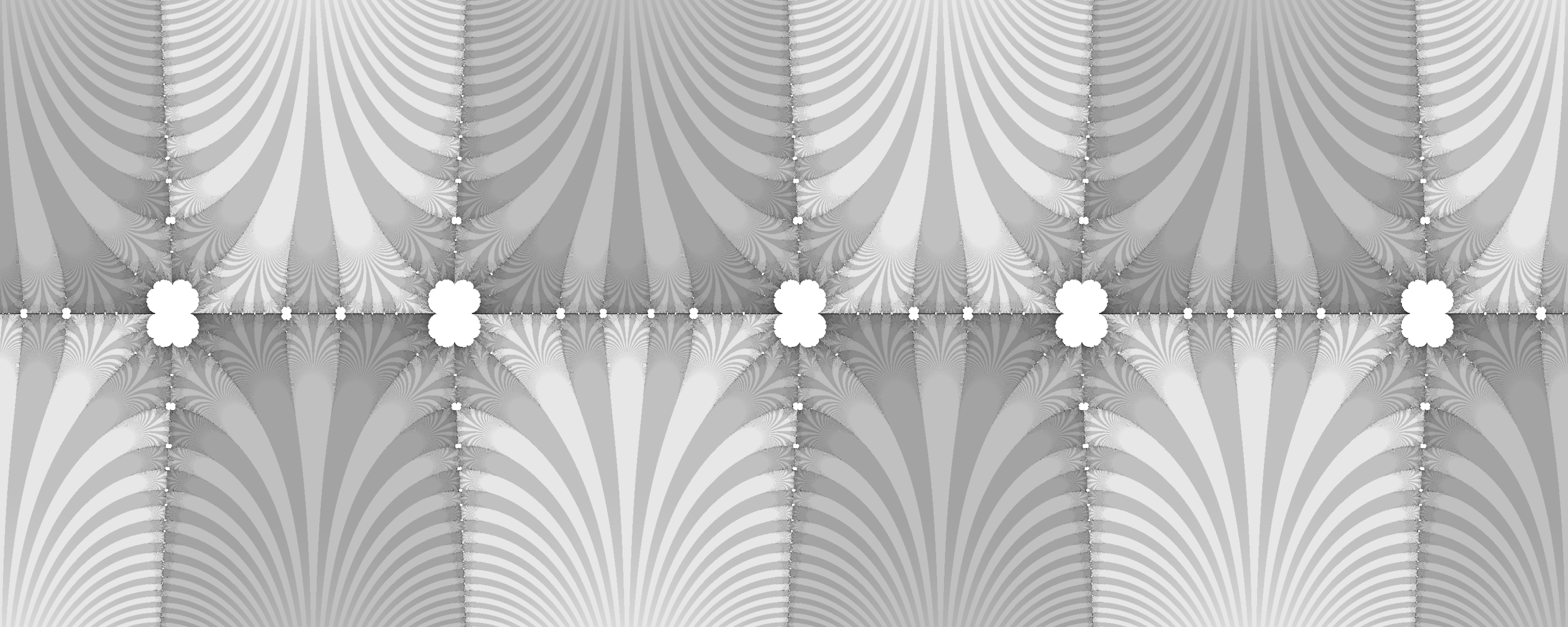}}
 \subfigure[$f_2(z)=z-1+e^{-z}+2\pi i$]{%
    \includegraphics[width=.49\textwidth]{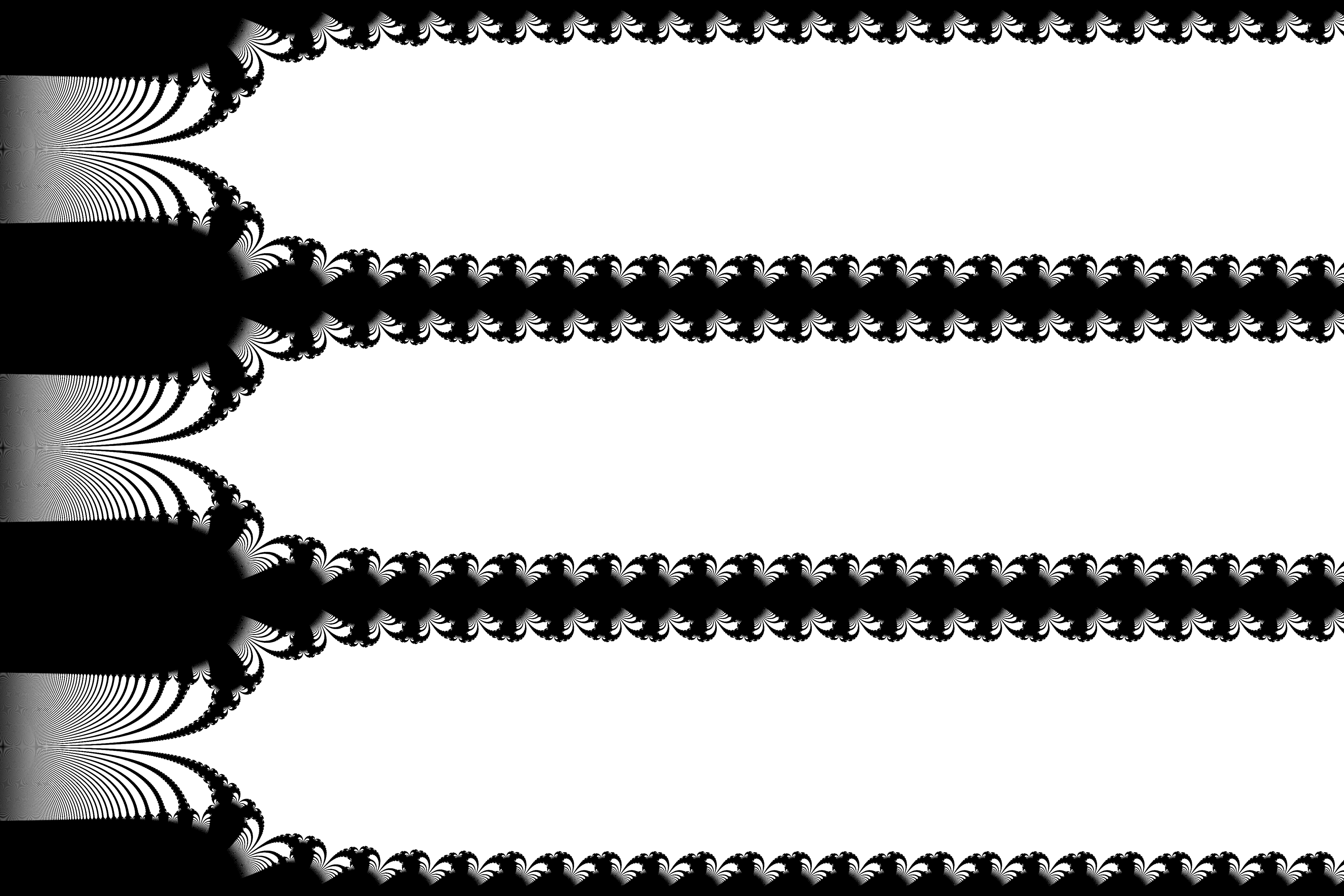}}\hfill
 \subfigure[$f_3(z)=\sinh(z)/z+2$]{%
    \includegraphics[width=.49\textwidth]{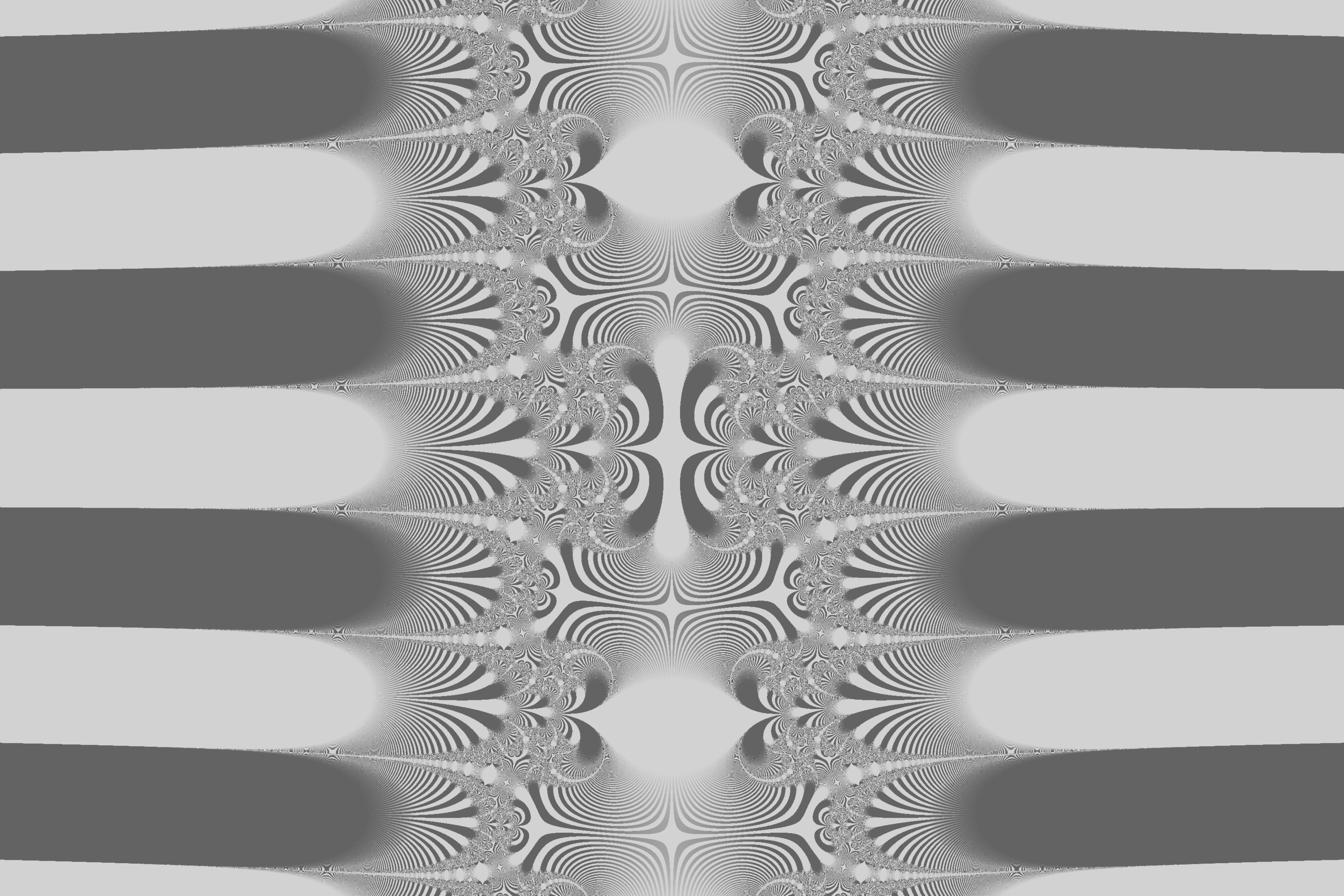}}
 \caption{\label{fig:juliasets}Each component of the
  Fatou sets of the functions $f_1$ and $f_2$
  (shown in white) is a wandering domain. On the other hand,
   our results show that
  $J(f_3)=\C$, as suggested by the picture. (For the maps
  $f_1$ and $f_3$, different gray-tones are used
  to highlight the structure within of the Julia set.)}
\end{center}
\end{figure}

 Quasiconformal deformation theory still appears to
  be essentially the only known method
  of proving the absence of wandering domains for general complex
  families of entire functions. 
  This method cannot be applied in most situations where $S(f)$ is infinite,
  and hence many problems are open in this setting. For example, a famous
  question 
  \cite[Problems 2.77, 2.87]{brannan_hayman} 
  asks whether an entire function can have a wandering domain with 
  a bounded orbit. 
    Indeed,
  for all explicit examples of wandering domains $U$, the iterates in $U$
  converge to infinity locally uniformly, although Eremenko and Lyubich
  \cite{elexamples} used approximation theory to construct examples
  with a finite accumulation point. 

 When the dynamical behaviour of the maps in question is restricted,
  other methods can be brought to bear. It is known that
  every limit point of an orbit in a wandering domain is a non-isolated point 
  of the 
  postsingular set $P(f)$ \cite[Theorem on p.\ 370]{bergweiler_etal}. 
  Furthermore, Eremenko and Lyubich \cite[Theorem 1]{eremenko_lyubich_2} showed
  that a function in the class
    \[ \B := \{f:\C\to\C \text{ transcendental entire}: 
                   S(f) \text{ is bounded} \} \]
   cannot have a wandering domain in which the iterates tend to infinity.

  Nonetheless, the settings in which these results can be applied to
   rule out wandering domains altogether tend
   to be rather restrictive, and many natural questions remain. For example,
   a well-known problem asks whether 
   a function in the Eremenko-Lyubich class $\B$
   can have wandering domains at all.\footnote{%
   Since this paper was first submitted, Bishop has announced
   the construction of such a function. Here the orbit of the wandering domain
   contains (infinitely many) critical values; hence 
   Question~\ref{question:main} remains open.}

  We propose the following variant of this question.

  \begin{question} \label{question:main}
   Let $f\in\B$, and suppose that the singular values of $f$ tend to infinity
    uniformly under iteration, that is,
     $\lim_{n\to\infty} \inf_{s\in S(f)} |f^n(s)| = \infty$.
    Can $f$ have a wandering domain? 
  \end{question}

  A function satisfying these hypotheses
   is given by $f(z) = \sinh(z)/z + a$, for
   $a$ sufficiently large. The question seems to have remained 
   open until now even for this simple example. The following result allows
   us to solve the problem in this case and many others. (Here and throughout, 
   $\dist$ denotes Euclidean distance.) 
    
 \begin{thm} \label{thm:mainescaping}
  Let $f\in\B$ be a function for which
   $f^n|_{S(f)}\to\infty$ uniformly.
   Let $A\subset \C$ be a closed set
   with $(S(f)\cup f(A))\subset A$  
   such that all connected components of $A$ are
   unbounded. 

  Suppose that there exist $\eps>0$ and $c\in(0,1)$ with the following property:
   if $z\in A$ is sufficiently large and $w\in\C$ satisfies
    $|w-z|<c|z|$, then 
    $\dist(f(w),S(f))>\eps$.

  Then $f$ has no wandering domains. 
 \end{thm}
\begin{remark}[Remark 1]
  For the function $f(z)=\sinh(z)/z + a$, with $a$ real and sufficiently large,
   the hypotheses are 
   satisfied for $A = [a,\infty)$. 
\end{remark}
\begin{remark}[Remark 2]
  Usually, we think of 
   the set $A$ as a union of one or finitely many
   curves to infinity (``hairs'' or ``dynamic rays'' as constructed in
   \cite{strahlen} for a large class of functions). Assuming the existence
   of such a set of curves, the 
   condition in the theorem then reduces essentially to the function-theoretic
   requirement
   that the tracts of the function $f$ over infinity are sufficiently
   thick along $A$. This condition tends to be automatically
   satisfied by explicit entire
   functions of finite order. 
\end{remark}
\begin{remark}
 We note that our result requires in no way that our functions 
  depend on only finitely many parameters, such as in Sullivan's theorem.
  Indeed, the functions
   $f(z) = \sinh(z)/z + a$ naturally belong to an 
   infinite-dimensional real-analytic parameter space. 
\end{remark}

 If $f\in\B$ is real---i.e., $f(\R)\subset\R$---with only real
  singular values, then it follows
  from the recent solution of the rigidity problem
  for real-analytic maps of the interval \cite{kss2,strienrigidity} that $f$
  cannot have a wandering domain with a bounded orbit 
  (compare \cite{rempe_vanStrien}.) We can combine
  this with our method to obtain a more complete picture for a large class
  of such functions. 

  \begin{thm} \label{thm:main}
   Let $f\in\B$ such that $S(f)\cup f(\R)\subset\R$. Furthermore,
    assume that there are constants $r,K>0$ such that
     \[ \frac{|f'(x)|}{|f(x)|} \leq K\cdot \frac{\log |f(x)|}{|x|} \]
     whenever $x\in\R$ with $|x|>r$ and $|f(x)|>r$. 
   Then $f$ does not have a wandering domain.
  \end{thm}
 \begin{remark}
  The analytic condition on $f$ is simply a reformulation of a geometric
   condition similar to the one in Theorem \ref{thm:mainescaping};
   compare Theorem \ref{thm:log_der}. 
 \end{remark}

 \begin{cor} \label{cor:sinh}
  Let $\lambda,a\in\R$ with $\lambda\neq 0$. Then
   the function $f(z)=\lambda\sinh(z)/z+a$ does not have wandering domains.    
 \end{cor}

 Note that our results give, in particular, a new proof of the fact that
  the Julia set of the exponential map $f(z)=\exp(z)$ is the entire 
  complex plane. This celebrated conjecture of Fatou was proved by Misiurewicz
  \cite{misiurewiczexp} in 1981, and can also be established by
  using Sullivan's argument (see \cite{bakerrippon,eremenko_lyubich_2}) or
  by appealing to the results of \cite{bergweiler_etal}. In
  Section \ref{sec:exp}, we present the argument
  in this special case separately, in order to illustrate and motivate
  the  idea underlying
  the proof. 

 Our results rely on a very general fact that relates to the comparison
  between hyperbolic metrics; see Theorem \ref{thm:expansion} for the
  precise statement. We note that this argument is somewhat similar in spirit
  to other applications of hyperbolic geometry in transcendental dynamics: These
  include some of
  Bergweiler's results \cite{bergweilerinvariantdomainssingularities} on
  the relation between the postsingular set and Baker domains, and
  Zheng's generalization of the results of \cite{bergweiler_etal} to 
  meromorphic and
  more general functions. For completeness,
  we discuss generalizations of the results of Bergweiler and Zheng in Section
   \ref{sec:further}. In particular, we prove the following result, 
   which can sometimes
   be used to exclude the existence of Baker domains of meromorphic functions.
   Recall that a \emph{Baker domain} is a component $W$
   of the Fatou set
   such that $f^{k}(W)\subset W$ and $f^{nk}|_W\to\infty$ for some 
   $k\geq 1$. 

 \begin{thm}[Baker domains of meromorphic functions] \label{thm:baker}
   Let $f:\C\to\Ch$ be a transcendental meromorphic function and suppose that
    $W$ is a Baker domain of $f$, with $k\geq 1$ as above.
    Then there exists a sequence $p_n$ of points in the postsingular
    set of $f$ such that 
      \[ |p_n|\to\infty, \quad 
          \sup \frac{|p_{n+1}|}{|p_n|} < \infty\quad\text{and}\quad
          \frac{\dist(p_n,W)}{|p_n|}\to 0. \]
 \end{thm}

 We also note that Theorem \ref{thm:expansion} could be used to exclude the
  presence of wandering domains also in situations quite different from
  those formulated in Theorems \ref{thm:mainescaping} and \ref{thm:main}. In
  particular, it also lends itself to applications to suitable families of
  transcendental meromorphic functions. 
  However, our  techniques do not 
   allow us
   to answer Question \ref{question:main} in general.

\subsection*{Notation}
 The complex plane and the unit disc are denoted by 
 $\C$ and $\D$, respectively; we use
 $\H_{>0}$ and $\H_{\geq 0}$ to denote the open resp.\ closed
 upper half plane.
 The Euclidean length of a rectifiable curve $\gamma\subset\C$ is denoted by 
 $\len(\gamma)$; for the Euclidean distance between a point $z$ and
  a set $A$, we write $\dist(z,A)$.
  We denote by 
  $D_r(z)$ the Euclidean disc of radius $r$ centred at the point $z$.

\smallskip

 One of the main tools in our proofs is given by \emph{plane hyperbolic
   geometry}. If an open set $U\subset\C$ omits at least two points,
   we denote the density of the hyperbolic metric on $U$ by
   $\rho_U$; i.e., $\rho_u(z)|dz|$ is the unique complete
   conformal metric of constant curvature $-1$ on $U$. 
   We shall briefly review some facts concerning the
   hyperbolic metric in Section \ref{sec:prel} and
   refer to \cite{beardonminda,keenlakic} for more details. 

\smallskip

Let $f:\C\to\C$ be a transcendental entire function. 
 The set of (finite) \emph{singular values} $S(f)$ of $f$
 is the smallest closed subset of $\C$ 
 such that $f:\C\setminus f^{-1}(S(f))\to\C\setminus S(f)$ 
 is a covering map. 
 It is well-known that $S(f)$ agrees with the closure of the set
 of critical and asymptotic values of $f$, but we shall not 
 require this fact. 
  The \emph{postsingular set} of $f$ is 
defined by 
   \[ P(f)=\cl{\bigcup_{n\geq 0} f^{n}(S(f))}. \]

Suppose that $W$ is a wandering domain. It is elementary to see
 that 
 every accumulation point $z_0$ of $f^n(w)$, for $w\in W$, belongs to
 $P(f)\cap J(f)$, where $J(f)=\C\setminus F(f)$ is the \emph{Julia set} of
  $f$. As mentioned in the introduction,
  $z_0$ is even known to be a non-isolated
  point of $P(f)$ \cite{bergweiler_etal}, but we shall not use this fact.

\subsection*{Acknowledgements} We would like to thank Walter Bergweiler, 
  J\"orn Peter and 
Phil Rippon for interesting
 discussions and helpful feedback.  We also thank the referee for comments that
 have led to some improvement in the paper. 

\section{The Julia set of the exponential} \label{sec:exp}
  As announced in the introduction, we shall motivate the proofs of our
   more general results by indicating a new, elementary and conceptual proof
   of Misiurewicz's theorem concerning the Julia set of the complex exponential function.
   The proof uses some basic facts of
   hyperbolic geometry; readers unfamiliar with this topic may 
   wish to refer to the short introduction provided at the beginning of Section \ref{sec:prel}.
 \begin{thm}[Misiurewicz, 1981]
  $J(\exp)=\C$.
 \end{thm}
 \begin{proof}
  Set $f := \exp$; we have $S(f)=\{0\}$. 
   Suppose, by contradiction, that $F(f)\neq\emptyset$. Let
   $w\in F(f)$ and let $W$ be the connected component of $F(f)$ containing
    $w$.
 \begin{claim}[Preliminary claim]
  $W$ is a simply connected wandering domain, and there is 
  a sequence $(n_k)$ with $f^{n_k}(w)\to 0$.
 \end{claim}
 \begin{subproof}
  This is classical, and was, at least mostly, already known to Fatou.
   (He states explicitly that $f$ does not have any attracting
    periodic orbits and no domains where the iterates converge
    to infinity \cite[p. 370]{fatou}). We provide a self-contained
   account for the reader's convenience. Firstly, every attracting
   or parabolic basin of an entire function must contain a singular value,
   but the orbit of $0$ tends to infinity. Similarly, there cannot be
   any rotation domains, since the boundary of such a component
   would need to be contained in the postsingular set.
   Hence $W$ is either a wandering
   domain, or is eventually mapped to a 
   Baker domain. In particular, every limit function of
   $f^{n}|_W$ is constant; i.e.,\ if $z_0\in \C\cup\{\infty\}$ and
   $f^{n_k}(w)\to z_0$, then $f^{n_k}\to z_0$ locally uniformly on $W$.

  Next we show that $f^{n}(w)\not\to\infty$ (in particular, $W$ is
   not a Baker domain). Indeed, otherwise
   $\re f^{n}(w)\to +\infty$. By replacing $w$ with its image
   under a suitable
   forward iterate, 
   we may assume that $\re f^n(w)\geq 2+2\pi$ for all $n\geq 0$. Now
   $|f'(z)|=|f(z)|$, and in particular
   $|f'(z)|\geq 2$ if $|f(z)-f^n(w)|\leq 2\pi$. 
   It follows that, for each $n\geq 0$, we can define
   a branch $\psi_n$ of $f^{-n}$, defined on $D_n := D_{2\pi}(f^n(w))$, such that
   $\psi_n(f^n(w))=w$ and $|\psi_n'(z)|\leq 2^n$ for all $z$. In particular,
   the size of $\psi_n(D_n)$ shrinks to zero as $n\to\infty$. On the other hand,
   $f(D_n)=f^{n+1}(\psi_n(D_n))$ intersects the negative real axis, and hence
   $f^2(D_n)$ intersects the unit disc. This
   contradicts the assumption that $f^n$ converges to infinity uniformly
   in a neighborhood of $w$. 

  So we have shown that the iterates in $W$ cannot converge to infinity
   locally uniformly; hence there is some $z_0\in\C$ with
   $f^{n_j}(w)\to z_0$ for a suitable sequence $(n_j)$. 
   By the maximum principle, this implies that 
   $W$ is simply connected. (Recall that, by Montel's theorem,
   there are preimages of the real axis, and hence points whose
   orbits converge to infinity, arbitrarily close to
   any point of $J(f)$.)

  Now let us assume that, if 
   $z_0=f^j(0)$ for some $j$, then $z_0$ is chosen such 
   that $j$ is minimal with this property. 

  We claim that we must have $j=0$. Indeed, otherwise,
   for small $\eps$ and all sufficiently large $k$,
   we can
   define a branch $\psi_k$ of $f^{-n_k}$ on $D_{2\eps}(z_0)$ that maps
   $f^{n_k}(w)$ to $w$. Since $z_0$ belongs to the Julia set, 
   we must have $\diam(\psi_k(D_{\eps}(z_0)))\to 0$
   (this uses Koebe's distortion theorem). We have obtained
   a contradiction to 
   the fact that $f^{n_k}\to z_0$ uniformly on a neighborhood of $w$. 
 \end{subproof}


 We now come to the new part of the argument.
  Set $U := \C\setminus [0,\infty)$; then 
  $f^n(W)\subset U$ for all $n$. Also set
  $\W := \bigcup_{n\geq 0} f^n(W) \subset U$. 
  The idea is as follows:
  \begin{enumerate}
    \item Since $f:f^{-1}(U)\to U$ is a covering map and $f^{-1}(U)\subset U$,
      the map $f$ expands the
      hyperbolic metric of $U$. 
      So the derivative $\delta_n$ of
      $f^n$ with respect to the hyperbolic metric of $U$ increases
      with $n$ by Pick's theorem (see Proposition \ref{prop:pick} below). \label{item:expansion}
    \item Since $f(\W)\subset\W$, the map $f$ contracts the hyperbolic
      metric of $W$ by Pick's theorem. Hence the derivative $\tilde{\delta}_n$ 
      of $f^n$ with respect to the
      hyperbolic metric of $\W$ decreases with $n$.
       \label{item:contraction}
    \item Near any point whose real part is sufficiently negative, an explicit
      calculation shows that
      the map $f$ \emph{strongly} expands the hyperbolic metric of $U$. 
           \label{item:strongexpansion}
  \end{enumerate}
  Observation (\ref{item:contraction}) implies that the sequence $\delta_n$
   remains bounded, while (\ref{item:strongexpansion}) implies that
   $\delta_n\to\infty$. (Recall that the orbit of $w$ accumulates at $0$,
   and hence enters any left half plane.) This
   yields the desired contradiction. 

 To make this argument precise, set $w_n := f^n(w)$; we have
  \begin{align*} \delta_n &= \|Df^n(w)\|_U := 
       |(f^n)'(w)|\cdot \frac{\rho_U(w_n)}{\rho_U(w)}\qquad\text{and}\\
         \tilde{\delta}_n &= \|Df^n(w)\|_{\W} := 
       |(f^n)'(w)|\cdot \frac{\rho_{\W}(w_n)}{\rho_{\W}(w)}.
  \end{align*}
 By (\ref{item:expansion}) and (\ref{item:contraction}), and
   since $\W\subset U$, these  quantities are related as follows:
   \begin{equation}\label{eqn:deltaprime}
    0 < \delta_0 \leq 
    \delta_n = \tilde{\delta}_n \cdot \frac{\rho_U(w_n)}{\rho_{\W}(w_n)}\cdot
                          \frac{\rho_{\W}(w)}{\rho_U(w)} \leq
           \tilde{\delta}_n \cdot \frac{\rho_{\W}(w)}{\rho_U(w)} 
    \leq \tilde{\delta_0}\cdot \frac{\rho_W(w)}{\rho_U(w)} =: \delta_0' < \infty. \end{equation}

  Let $n_k$ be as in the preliminary 
   claim; i.e.\ $w_{n_k}\to 0$,
    and hence $\re w_{n_k-1}\to -\infty$. In order to show
    that $f$ strongly expands the hyperbolic metric of $U$ near
    $w_{n_k-1}$, we use 
     two simple estimates on $\rho_U$: 
 \[   \rho_U(z) \geq \frac{1}{2|z|}\quad\text{for all $z\in U$}\quad\text{and}\quad
    \rho_U(z) \leq \frac{2}{|z|}\quad\text{when $\re z<0$} \]
   These follow from the standard estimate
    on the hyperbolic metric in a simply connected domain
    (Proposition \ref{prop:Koebe_hyp_metric}), or alternatively 
    directly from the
    explicit formula
   \[ \rho_U(z) = \frac{1}{2|z|\cdot \sin(\arg(z)/2)}. \] 
    
   Now suppose that $z\in f^{-1}(U)$ with $\re z < 0$. Then we have
     \[ \|Df(z)\|_U =
        |f'(z)|\cdot \frac{\rho_U(f(z))}{\rho_U(z)} \geq
        \frac{1}{4}\frac{|f'(z)|}{|f(z)|}\cdot |z| = \frac{|z|}{4} \geq 
        \frac{|\re z|}{4}. \] 

   Suppose that
   $k$ is sufficiently large that $\re w_{n_k-1}< -4\delta_0'/\delta_0$. Then
   \[
      \delta_{n_k} = \delta_{n_k-1}\cdot \|Df(w_{n_k-1})\|_U  \geq
                \delta_{0}\cdot \frac{|\re w_{n_k-1}|}{4} >
                \delta_0'. \]
   This contradicts \eqref{eqn:deltaprime}, and the proof is complete.
 \end{proof}
\begin{remark}
 We have presented the proof in the present form to emphasise the ideas that will occur
  in our more general results. It is in fact possible to simplify the argument yet further,
  eliminating the need for the preliminary claim and yielding
  a proof that altogether avoids
  the classification of periodic Fatou components,
  the Riemann mapping theorem, Montel's theorem
  and Koebe's theorems. 
  All that is needed are formulae for the hyperbolic metric
  on a strip and a slit plane, together with Pick's theorem (i.e., the Schwarz lemma
  together with the invariance of the hyperbolic metric under automorphisms of the disc).

 The proof can be sketched as follows. Suppose, by contradiction, that
  there is a point  $w$ at which the family of iterates is
  equicontinuous. First consider the case where $w$ does not tend to infinity 
  under iteration.
  Then there is a small open disc $D$ around $w$ consisting of nonescaping
   points, and hence satisfying 
   $f^n(D)\subset U:=\C\setminus [0,\infty)$ 
   for all $n$. Considering the increasing sequence of 
  hyperbolic derivatives 
   \[ \|Df^n(w)\|_U^D := |(f^n)'(w)|\cdot \frac{\rho_D(w)}{\rho_U(f^n(w))}\leq 1, \]
   we see first that no limit point
  of the orbit of $w$ can belong to $U$. Just as above, it then follows that
  there can be no limit point in $[0,\infty)$ either, a contradiciton.

  If, on the other hand,
   $f^n(w)\to\infty$, then we can consider $U' := \C\setminus (-\infty,0]$, and
   derive a contradiction in analogous fashion.
\end{remark}

\section{Estimates on the hyperbolic metric}
\label{sec:prel}

\subsection*{A brief review of hyperbolic geometry}
 Throughout this paper, it will be convenient, in somewhat nonstandard notation,
   to use the term 
   \emph{Riemann surface} for any space $X$ that is locally homeomorphic to 
   open subsets of the complex plane,
   with analytic bonding maps, without requiring that $X$ is connected. 
   That is, for us 
   a Riemann surface will be a nonempty union of \emph{connected} Riemann surfaces.
   In particular, any non-empty open subset of the Riemann sphere is a Riemann surface. 

 A Riemann surface $U$ is called \emph{hyperbolic} if there exists a 
 holomorphic
 universal covering map $\pi:\D\to U$;
 if $U\subset\C$, then $U$ is hyperbolic 
 if and only if $\C\setminus U$ contains at least two points.

 As stated in the introduction, 
  the hyperbolic metric on a hyperbolic surface $U$ is  
 the unique complete conformal Riemannian metric on $U$ of constant curvature 
 $-1$. Equivalently, it is the local push-forward of the 
  hyperbolic metric on $\D$ by the covering map
  $\pi$. (This is well-defined because the hyperbolic metric on $\D$ is 
  invariant under
   M\"obius transformations.) 
 Recall that we
  denote this metric by $\rho_{U}(z)\vert dz\vert$ for plane domains; 
  we can do so also for hyperbolic surfaces, but here the function $\rho_U$ 
  will depend on
  the choice of a local coordinate $z$.
  If $U$ is not connected and every component of $U$ is hyperbolic, then
  the hyperbolic metric on $U$ is defined componentwise.

The hyperbolic length of a rectifiable curve $\gamma\subset U$ will be denoted
 by $\ell_U(\gamma)$.
 For any two points $z,w\in U$, the 
 hyperbolic distance $d_{U}(z,w)$ 
 is the smallest hyperbolic length of a curve connecting 
  $z$ and $w$ in $U$. 
 In order to be able to discuss the derivatives of a holomorphic
  function, and the relative densities of hyperbolic metrics, we 
  introduce the following notation.

\begin{defn}[Hyperbolic derivatives and relative densities]
 Let $V$ and $U$ be hyperbolic Riemann surfaces, and let $z\in U$. Suppose that
   $f$ is a holomorphic function defined on a neighborhood of $z$ and taking values
   in $V$. Then we denote the hyperbolic derivative of $f$ with respect to the metrics
   in $U$ and $V$ by
   \[   \|Df(z)\|_V^U := |f'(z)|\cdot \frac{\rho_V(f(z))}{\rho_U(z)}.  \]
   (Note that this quantity is independent of local coordinates.) If both
    $z$ and $f(z)$ belong to $U$, we also abbreviate 
    $\|Df(z)\|_U := \|Df(z)\|_U^U$. 
 If $V\subset U$, then we define
   \[\rho_V^U(z) = \frac{1}{\|D\iota(z)\|_U^V} = 
    \|D\iota^{-1}(z)\|_V^U = \frac{\rho_V(z)}{\rho_U(z)}, 
   \] 
  where $\iota:V\to U$ is the inclusion map $\iota(z)=z$. 
  In other words, 
   $\rho_V^U(z)$ is the relative density of the hyperbolic
   metric of $V$ with respect to the hyperbolic metric of $U$. 
\end{defn}

A key fact (already used in the previous section)
  is given by Pick's theorem, also referred to as the Schwarz-Pick lemma 
 \cite[Theorem 10.5]{beardonminda}.

\begin{prop}[Pick's Theorem] \label{prop:pick}
 A holomorphic map $f:U\to V$ 
 between two hyperbolic
 Riemann surfaces does not increase the respective hyperbolic metrics, i.e., 
 $d_V(z,w)\geq d_U(f(z),f(w))$ holds for all $z,w\in V$, or equivalently 
 $\| Df\|_V^U \leq 1$ for all $z\in V$. 
 Equality holds in the latter inequality if and 
  only if $f$ is a covering map; in this case, we say that $f$ is a local 
  isometry. 
  Otherwise, $f$ is a strict contraction.
  
  In particular, if $V\subsetneq U$, 
  then $\rho^U_V(z) > 1$ for all
  $z\in V$. 
\end{prop}

\subsection*{Formulae and estimates for the hyperbolic metric}
  It is well-understood how the hyperbolic metric on the Euclidean domain
    $U$ depends on the
   geometry of $U$. We briefly collect the results relevant for
   our proofs. 
   Recall that
   the hyperbolic metrics on the unit disc, on the
   upper half plane $\H_{>0}$ and on the punctured disc
   $\D^*=\D\setminus\{0\}$ are given by
   \[ \rho_{\D}(z)= 
       \frac{ 2}{1-|z|^2}, \qquad
        \rho_{\H_{>0}} =
         \frac{1}{\im(z)}\qquad\text{and} \qquad
        \rho_{\D^*}(z) =
      \frac{1}{|z|\cdot |\log|z||}. \]

 The following estimates are useful for general domains;
  they follow from the Schwarz lemma and Koebe's $1/4$-theorem, respectively.

\begin{prop}(\cite[Theorems 8.2 and 8.6]{beardonminda})
\label{prop:Koebe_hyp_metric}
 Let $U\subset\C$ be a hyperbolic domain. Then
\[
  \rho_U(z)\leq  \frac{2}{\dist(z,\partial U)}.
\]
 If $U$ is simply connected, then also
\[ 
\rho_U(z) \geq \frac{1}{2\dist(z,\partial U)}. \]
\end{prop}

\subsection*{Relative densities of hyperbolic metrics}
Now suppose that $U$ and $V$ are hyperbolic domains with
 $V\subsetneq U$; by Pick's theorem, 
 the inclusion $\iota$ from $V$ into $U$ is contracting: 
 $\rho^U_V(z)>1$ for all $z\in V$. We now state and prove 
 a proposition
 that relates 
 the strength or weakness of this expansion to the
 hyperbolic distance between $z$ and the boundary of $V$.
 (Compare also \cite[Lemma 3.1]{schleicher_zimmer}.) This result is likely
 not new, but as we are not aware of a reference in the literature,
 we shall include the simple proof. 

\begin{prop}
\label{prop_metric_comp}
Let $V$ and $U$ be hyperbolic Riemann surfaces with $V\subsetneq U$.
 Let $z\in V$ and set $R := d_U(z,U\setminus V)$. Then 
\begin{align*}
  1<\frac{2e^R}{(e^{2R}-1)\cdot
    \log\frac{e^R+1}{e^R-1}} \leq
      \rho^U_V(z) \leq 1 + \frac{2}{\e^{R} -1}.
\end{align*}
\end{prop}
\begin{remark}
 The exact dependence of the bounds (which are sharp, as the proof 
  will show) on the number $R$ 
  is not relevant for our
  purposes. What is important is that they depend only on $R$ and that the upper bound tends to
  $1$ as $R\to\infty$, while the lower bound tends to infinity as $R\to 0$.  
\end{remark}
\begin{proof} 
Let $\pi: \D\to U$ be a universal covering map with $\pi(0)=z$, and let
  $\Vt$ be the component of $\pi^{-1}(V)$ containing $0$. Since
 $d_{\D}(x,y)\geq d_{U}(\pi(x),\pi(y))$ 
 for all $x,y\in \D$, the set $\Vt$ contains a hyperbolic disc centred at
  zero and of radius $R$, measured with respect 
  to
  the hyperbolic metric on $\D$. 

 Furthermore, there is a point of $\partial V$ that has
  hyperbolic distance $R$ from $z$ (in the metric of $U$). Connecting
  this point to $z$ using a geodesic of length $R$ and lifting this
  geodesic to the unit disc, we see that $\partial \Vt$ contains 
  a point $\wt{R}$ such that $\dist_{\D}(0,\wt{R})=R$. 
  By precomposing with a rotation, 
  we may normalise $\pi$ such that 
  $\wt{R}$ is real and positive. 
  (We emphasize that it does not matter in this argument
   whether or not $\pi$ is injective on the hyperbolic disk of radius $R$,
   since geodesics of $U$ always lift to geodesics of $\D$.)

 Hyperbolic discs in $\D$ centred at $0$ are Euclidean discs centred
  at $0$, hence we have
  \[ D_{\wt{R}}(0)\subset \Vt \subset \D\setminus\{\wt{R}\}. \]
  Thus we obtain upper and lower bounds:
   \[ \rho_V^U(z) = \frac{\rho_{\Vt}(0)}{\rho_{\D}(0)} \leq
      \frac{\rho_{D_{\wt{R}}(0)}(0)}{\rho_{\D}(0)} = 
      \frac{1}{\wt{R}} \]
    and likewise
   \[ \rho_V^U(z) \geq
      \frac{\rho_{\D\setminus\{\wt{R}\}}(0)}{\rho_{\D}(0)} =
      \frac{\rho_{\D^*}(\wt{R})}{\rho_{\D}(\wt{R})}. \]
   So overall we see that
   \begin{equation} \label{eqn:nonexplicitbounds}
    1 < \frac{\rho_{\D^*}(\wt{R})}{\rho_{\D}(\wt{R})} \leq
       \rho_V^U(z) \leq 
            \frac{1}{\wt{R}}, \end{equation}
   and the bounds are sharp. As mentioned in the remark above, these inequalities 
    would be
    sufficient for the purposes of
    our article,
       together with the observation that
    $\wt{R}\to 1$ as $R\to\infty$. 

  To obtain the explicit bounds as stated, we compute the
   values in (\ref{eqn:nonexplicitbounds}). First note that
   the hyperbolic distance in $\D$ between $0$ and $z\in\D$ is
   given by $d_{\D}(z) = \log\frac{1+|z|}{1-|z|}$, and thus 
   $\wt{R}=\frac{e^R-1}{e^R+1}$. Hence 
    $1/\wt{R}=1+2/(e^R-1)$, proving the upper bound. 
   For the lower bound, we apply the
    explicit formulae for the hyperbolic densities of $\D$ and $\D^*$:
\[      \frac{\rho_{\D^*}(\wt{R})}{\rho_{\D}(\wt{R})}
    = \frac{1-\wt{R}^2}{2\wt{R}\cdot|\log\wt{R}|} 
     = \frac{2e^R}{(e^R+1)^2}\cdot\frac{e^R+1}{e^R-1}\cdot\frac{1}{\log\frac{1}{\wt{R}}} 
     = \frac{2e^R}{(e^{2R}-1)\cdot\log\frac{e^R+1}{e^R-1}},  \] 
  as claimed.
 \end{proof}

The following is an immediate consequence of  
Proposition \ref{prop_metric_comp}.
\begin{cor}
\label{cor_metric_comp}
 Let $U$ be a hyperbolic Riemann surface, let $V\subsetneq U$ be an open subset
  of $V$ and let $(z_n)$ be a sequence of points in $V$. 
  Then $\dist_U(z_n,U\setminus V)\to\infty$ if and only if 
\begin{align*}
 \rho_V^U(z_n)\searrow 1\text{ as } n\to\infty.
\end{align*}
\end{cor}

\section{Statement and proof of the main result}
 We now state our main technical result, which allows us to
  exclude the existence of wandering domains in many cases. Although we mainly apply
  the theorem in the case of an entire function, 
  we shall state it generally for holomorphic maps
  on arbitrary Riemann surfaces. 

 \begin{thm}
\label{thm:expansion}
 Let $U$ be a hyperbolic Riemann surface. Let $U'\subset U$ be open and let
   $f:U'\to U$ be a holomorphic covering map. 
   Assume that there is an open connected set 
   $W\subset U'$ such that 
   $f^n(W)\subset U'$ for all $n\geq 0$. 

  Then, for all $w\in W$,
    \begin{equation} \label{eqn:zhengestimate}
    \liminf_{n\to\infty}
         d_U(f^{n}(w), U \setminus f^{n}(W))>0. \end{equation}

  Furthermore, let $D\subset U$ be open and set $V:= f^{-1}(D)$. Suppose $(n_k)$
   is a sequence such that 
   $f^{n_k}(w)\in D$ and $\dist_U(f^{n_k}(w), U\setminus D)\to\infty$. Then 
   \begin{equation} \label{eqn:ourestimate}
     d_U(f^{n_k-1}(w),U\setminus V)\to \infty.
  \end{equation}
 \end{thm}
\begin{remark}[Remark 1]
 The statement is a little technical, so the reader may wish to connect it
  with our argument in the case of the exponential map.
  Here $U=\C\setminus [0,\infty)$, 
   $f(z)=e^z$ and $U'=f^{-1}(U)$. 

  Let $D=U\cap D_{\eps}(0)$ for some $\eps>0$. Then  $V$
   is contained in the left half plane $\{\re z <\ln(\eps)\}$, and the hyperbolic
   distance in $U$ between any point of this half plane and the boundary
   $\{\re z = \ln(\eps)\}$ is bounded from above. (To see this, connect the
   given point with the boundary using a segment of a large semi-circle
   centred at zero.)
  
  If $W$ is a wandering domain whose orbit accumulates at zero, then
    we can apply (\ref{eqn:ourestimate}) to obtain a contradiction.
   Hence such wandering domains cannot exist 
   (and thus there are no wandering domains at all).
\end{remark}
\begin{remark}[Remark 2]
 The inequality (\ref{eqn:ourestimate}) 
   will be the main ingredient in the proof of Theorems
   \ref{thm:mainescaping} and \ref{thm:main}. Inequality (\ref{eqn:zhengestimate}) 
   is not required in these proofs, but arises naturally from our considerations;
   we state it here mainly to highlight the connection with previous work.
   Indeed, (\ref{eqn:zhengestimate})
   appears implicity (in a less general setting, but with the same idea in the proof) 
   in work of 
   Zheng \cite[Proof of Theorem 2.1]{zhengahlfors}, and indeed we can
   generalise \cite[Theorem 2.1]{zhengahlfors} to a larger class of
   functions using (\ref{eqn:zhengestimate}); see Proposition
   \ref{prop:bergweileretal}. Furthermore, in the
   case where $W$ is chosen maximal with the given properties (e.g.\ if
   $W$ is a Fatou component of a transcendental entire or
   meromorphic function), the distance in 
   (\ref{eqn:zhengestimate}) actually tends to infinity. This is the type of
   argument used by Bergweiler 
   in 
   \cite[Lemma 3]{bergweilerinvariantdomainssingularities}, and will be used
   in our proof of Theorem \ref{thm:baker}. Since these extensions 
   of Theorem \ref{thm:expansion} depart from our main line of inquiry, we shall defer
   their discussion to Section \ref{sec:further}. 
\end{remark}
\begin{remark}[Remark 3]
 It seems natural to state our conclusions in terms of the hyperbolic distance.
   By Proposition \ref{prop_metric_comp}, we can equivalently phrase them as results on the
   relative density of
   the hyperbolic metric of $f^n(W)$ and $V$ with respect to the hyperbolic metric on $U$:
   \begin{equation}
     \limsup_{n\to\infty} \rho^U_{f^n(W)}(f^n(w)) < \infty \end{equation}
    and
   \begin{equation}
      \rho^U_V(f^{n_k-1}(w)) \to 1. \end{equation}
  Indeed, it is these conditions that our proof will establish. 
\end{remark}
 \begin{proof}
 The proof begins just as in the case of the exponential map.
  We define $w_n := f^n(w)$, $W_n:= f^n(W)$ and study the
  sequences 
    \[ \delta_n := \|Df^n(w)\|_U \geq 1 \quad\text{and}\quad
       \tilde{\delta}_n := \|Df^n(w)\|^W_{W_n} \leq 1. \]
  Because $f$ is a covering map, the first of these sequences is nondecreasing, while the
  second is nonincreasing by Pick's theorem. The sequences satisfy 
  \[ 0 < \delta_0 \leq \delta_n \leq C\cdot \tilde{\delta}_n \leq C\cdot \tilde{\delta}_0 =: \delta_0', \]
  where $C=\rho^U_{W}(w)$. (I.e., $1/C$ is the hyperbolic derivative at $w$
   of the inclusion of $W$ into $U$.) 
    In particular, we have
   \[ \rho^U_{W_n}(w_n) =
       C\cdot \frac{\tilde{\delta}_n}{\delta_n} \leq C. \]
    As noted in Remark 3, this implies (\ref{eqn:zhengestimate}) by Proposition
     \ref{prop_metric_comp}. 

   Let us set $\eta_n := \|Df(w_n)\|_U$. Then
    $\delta_{n+1}= \eta_n\cdot \delta_n$, and thus $\eta_n\to 1$ 
    as $n\to\infty$.
  Now let $D$ and $n_k$ be as in the statement of the theorem. 
   Then $w_{n_k-1}\in V =f^{-1}(D)$, for all $k\geq 0$. 
    We write
  \[ 
      \eta_{n_k-1} = \|Df(w_{n_k-1})\|_U =
         \frac{\rho^U_{V}(w_{n_k-1})}{\rho^U_{D}(w_{n_k})}\cdot 
         \|Df(w_{n_k-1})\|^V_{D} = 
    \frac{\rho^U_{V}(w_{n_k-1})}{\rho^U_{D}(w_{n_k})}  \]
   where we used the fact that $f:V\to D$ is a covering map, and hence
   a local isometry. 
   By assumption, we have
    $d_{U}(w_{n_k},U\setminus D)\to\infty$ as $k\to\infty$; hence
    $\rho^U_{D}(w_{n_k})\to 1$ by Corollary
    \ref{cor_metric_comp}. Thus we see that
   \[ \rho^U_V(w_{n_k-1}) =
       \eta_{n_k-1}\cdot \rho^U_{D}(w_{n_k}) \to 1. \]
    This statement is
     equivalent to (\ref{eqn:ourestimate}) by Corollary \ref{cor_metric_comp}. 
\end{proof}

In the next section, 
 we apply the preceding result in the following setting: 
\begin{itemize}
\item $f:\C\to\C$ is a transcendental entire function, 
\item  $U=\C\setminus A$, where $A$ is a closed forward-invariant set that contains $S(f)$,
\item $W$ is a wandering domain of $f$ whose orbit is disjoint from $A$, 
\item $D=U\cap D(z_0,\eps)$, where $z_0$ is a finite limit function of the sequence $f^n|_W$. 
\end{itemize}
We note that, in particular, we can let $A$ be the postsingular set of $f$.
  Since $z_0$ belongs to the Julia set of $f$, the estimate
  (\ref{eqn:zhengestimate}) implies that $z_0\in P(f)$. Furthermore, 
  $W$ must be simply connected, and it is easy to deduce from this and
  (\ref{eqn:zhengestimate}) that $z_0$ cannot be an isolated point of
  $P(f)$; compare the claim in the proof of Theorem \ref{thm:zhengbergweiler} below. 
  This fact was originally proved in \cite{bergweiler_etal} for entire functions
  using a different method; our argument is in essence
  the same as that given by Zheng \cite{zhengahlfors} for functions meromorphic
  outside a small set.

\section{Application to entire functions}

We now use Theorem \ref{thm:expansion}
to deduce
 Theorem \ref{thm:mainescaping} and a slightly different form of
 Theorem \ref{thm:main}. (In the next section, we shall see that this 
 form implies the original formulation.) 

\subsection*{Accumulation at singular values through unbounded sets}
  We begin by showing that, in the situations we consider, 
  wandering domains must accumulate
  on singular values.

\begin{lem}
\label{lem:singvalue}
 Let $f\in\B$. Assume that there is a number $R>0$ such that
  the iterates of $f$ tend to infinity uniformly on
  the set $\{z\in P(f): |z|\geq R\}$. 

  Let $W$ be a wandering domain of $f$ for which $\infty$ is a limit function,
   and let $w\in W$. Then there is $s\in S(f)$ and a 
   sequence $(n_k)$ such that
    \[ f^{n_k}(w)\to s\quad\text{and}\quad f^{n_k-1}(w)\to\infty. \]
\end{lem}
\begin{proof}
  We may assume without loss of generality that $R$ is chosen so
   large that 
   $|s|<R$ for all $s\in S(f)$. 
 \begin{claim}[Claim 1]
  For every $K>R$, there exists $N$ such that
    \[ f^{-n}(\overline{D_K(0)})\cap P(f)\subset D_R(0) \]
   for all $n\geq N$.  
  In particular, the set
    \[ P(f)\cap 
       \bigcup_{m=0}^{\infty} f^{-m}( D_K(0)) \]
    is bounded.
 \end{claim}
 \begin{subproof}
   Let us fix some $K>R$. 
   By assumption, there is $N$ such that
   \[ f^{n}(P(f)\setminus D_R(0))\cap \cl{D_K(0)} = \emptyset \]
    for $n\geq N$. 
    This proves the first statement.

   To deduce the second statement, set 
     \[ M := \max\{|f^j(z)|: |z|\leq R, 0\leq j \leq N\}, \]
    let $v_0\in P(f)$ with $|v_0|>M$ and let $m\geq 0$. We will show that
    $|f^m(v_0)|\geq K$. Note that this is true by
    construction for $m\geq N$, so we may suppose that $m<N$. 

    Since $P(f)$ is the closure of the union
    of iterated forward images of singular values of $f$, 
    there is a sequence
    $v_k\in P(f)$ with $v_k\to v_0$, $|v_k|>M$ and
    $v_k=f^{n_k}(s_k)$, where $s_k\in S(f)$. By choice of $M$, we must 
    have $n_k>N$ (since 
    $|s_k|<R$) and  
    $|f^{n_k-N+m}(s_k)|>R$ (since $m<N$) for all $k$.
    Thus it follows
    from the first statement that
    \[ |f^m(v_k)|= |f^{N}(f^{n_k-N+m}(s_k))| > K. \] 
    By continuity, we see that $|f^m(v_0)|\geq K$, as claimed.
 \end{subproof}

 Now let $W$ and $w$ be as in the statement of the lemma.

 \begin{claim}[Claim 2]
   There is a sequence $n_k\to\infty$ such that
    $f^{n_k-1}(w)\to\infty$ and 
    $f^{n_k}(w)\to s$ for some $s\in P(f)$. 
 \end{claim}
\begin{subproof}
  As mentioned in the introduction, Eremenko and Lyubich
   \cite[Theorem 1]{eremenko_lyubich_2} proved that functions in $\B$ cannot
   have Fatou components in which the iterates tend to infinity.
   Hence we may choose $K\geq R+1$ sufficiently large that
   $\liminf_{n\to\infty} |f^{n}(w)|< K$. 
   Recall also that, by assumption, 
   $\limsup_{n\to\infty} |f^{n}(w)| =\infty$.

 By Claim 1, there is $M>K$ such that
   $|z|< M$ whenever $z\in P(f)$ with $|f^j(z)|\leq K$ for some 
   $j\geq 0$.

 Let $m_k\to\infty$ be a sequence such that
   $|f^{m_k}(w)|> M$. For every $k$, choose $p_k\geq m_k$ minimal with
   $|f^{p_k}(w)|\leq K$, and let $n_k\in \{m_k+1,\dots,p_k\}$ be minimal with
   $|f^j(w)|\leq M$ for $n_k\leq j\leq p_k$. 

  We first claim that $p_k-n_k\leq N$ for sufficiently large $k$, where
   $N$ is as in Claim 1. Indeed, otherwise we may assume, passing
   to a subsequence, if necessary, that 
   $p_{k}>n_{k}+N$ for all sufficiently large $k$. Then 
   $\vert f^{p_k-N}(w)\vert \leq M$. Let $v_1$ be an accumulation point
   of the sequence $f^{p_k-N}(w)$. Then $v_1\in P(f)$ and 
   $|v_1|\leq M$. By continuity, we then have 
   $|f^{N}(v_1)|\leq K$, and hence, by Claim 1,
   $|v_1|\leq R\leq K-1$. But then
   $|f^{p_{k}-N}(w)|\leq K$, which is
   a contradiction to the minimality of $p_{k}$. 

  We claim that $f^{n_k-1}(w)\to\infty$. Indeed, otherwise
   let $v_2\in P(f)$ be a finite accumulation point of this sequence.
   Since $\vert f^{n_k-1}(w)\vert >M$, we have that $|v_2|\geq M$. 
   By continuity, $|f^{m}(v_2)|\leq K$ for some $m\leq N$. But this
   is a contradiction to the choice of $M$. 

 The claim follows, replacing $(n_k)$ by a subsequence for which
   $f^{n_k}(w)$ is convergent, if necessary.   
\end{subproof}

\begin{claim}[Claim 3]
  If $s$ is as in Claim 2, then $s\in S(f)$. 
\end{claim}
\begin{subproof}
  Let $\eps>0$ be sufficiently small. Then it follows from Claim 1
   that the union of all components of
   $V := f^{-1}(D_{\eps}(s))$ that intersect $P(f)$ is bounded. 
  In particular, for sufficiently large $k$, the component
   of $V$ containing $f^{n_k-1}(w)$ does not intersect $P(f)$. 

  If $s$ is not a singular value, it follows that we can define
   a branch of $f^{-n_k}$ on $D_{\eps}(0)$ that maps 
   $f^{n_k}(w)$ to $w$. This leads to a contradiction since $s$ belongs
   to the Julia set. 
\end{subproof}
\end{proof}

\begin{proof}[Proof of Theorem \ref{thm:mainescaping}]
  Assume that the hypotheses of Theorem \ref{thm:mainescaping} are satisfied. That is,
   $f\in\B$ is a function for which the singular values escape to infinity uniformly and 
   $A\subset\C$ is a closed set with $(S(f)\cup f(A))\subset A$ such that
   all 
   connected components of $A$ are unbounded. Furthermore, there are $\eps>0$
   and 
   $c\in (0,1)$ such that the following holds: If $z\in A$ with $|z|\geq R$ and
   $w\in\C$ with $|w-z|\leq c|z|$, then $\dist(f(w),S(f))>\eps$. 
  We 
   set $U := \C\setminus A$; then every component of $U$ is simply connected.

  Assume, by contradiction, that $f$ has a wandering domain $W$ 
   and let $w\in W$. Then every finite accumulation point of $f^n(w)$ is contained in
   $P(f)$, and in particular $f^n(w)$ accumulates at $\infty$. By the preceding lemma,
   there is a singular value $s\in S(f)$ and a sequence $n_k$ such that
   $f^{n_k}(w)\to s$ and $f^{n_k-1}(w)\to \infty$. 
  Disregarding finitely many
   entries, we may assume that
   $f^{n_k}(w)\in D_{\eps}(s)$ for all $k$. 

  In particular, $w\notin A$ (otherwise, the hypotheses would imply that
   $\dist(f^{n_k}(w),s)>\eps$ for all sufficiently large $k$). 
   Since $w\in W$ was arbitary,
   this means that $W\subset U$. 

  \begin{claim}
    Set $U:=D_{\eps}(s)\cap U$ and $V:= f^{-1}(D)$, and let 
     $z_0\in V$ with $|z_0|\geq 2 R$. Then
     $\dist(z_0,\partial U)\geq c|z_0|/2$.
   In particular, $d_U(z_0,U\setminus V)\leq 8\pi/c$. 
  \end{claim}
 \begin{subproof}
  Let $u\in \partial U$. If $|u|<|z_0|/2$, then
    $|z_0-u|\geq |z_0|/2\geq c|z_0|/2$.
   Otherwise, we have $|u|\geq \vert z_0\vert /2 \geq R$. 
   Since $\dist(f(z_0),S(f))<\eps$,
    the hypotheses of the theorem imply that
   \[ |z_0-u|\geq c\cdot |u| \geq c|z_0|/2. \]

   For the second statement, simply let
    $\gamma$ be an arc of the circle
    $\partial D_{|z_0|}(0)$ that connects $z_0$ to $\partial V$.
    (Note that every component of $V$ is simply connected, and hence
     such an arc must exist. Indeed,
     every component of $D$ is simply connected 
     because $A$ has no bounded components, and $f:V\to D$ is a covering map.) 
    By the previous estimate, 
    $\dist(z,\partial U)\geq c\vert z\vert/2 = c\vert z_0\vert/2$
    for all $z\in\gamma$. 
    Using Proposition \ref{prop:Koebe_hyp_metric}, we see that 
     \[ d_U(z_0,\partial V)\leq
       \ell_U(\gamma) \leq 
     2\pi \vert z_0\vert\cdot \max_{z\in\gamma}\rho_U(z) \leq 
     2\pi \vert z_0\vert\cdot\frac{4}{c\vert z_0\vert}
   =\frac{8\pi}{c}.\qedhere\]
 \end{subproof}

 In particular, we see that $d_U(f^{n_k-1}(w),U\setminus V)$ stays bounded
  as $k\to\infty$ while $d_U(f^n(w),U\setminus D)\to\infty$. 
  This contradicts Theorem \ref{thm:expansion}. 
\end{proof}

\subsection*{The case of real functions} 

We now prove a version of Theorem \ref{thm:main} by combining our
 method with rigidity results for real functions, as
 discussed in \cite[Section 3]{rempe_vanStrien}. 

\begin{defn}
We denote by $\B_{\real}$ the set of all real transcendental entire functions with bounded
sets of singular values. 
The set $\B_{\real}^{*}$ consists of all maps in $\B_{\real}$ with real singular values.
\end{defn}

 We will use the following rigidity result for the class
  $\B_{\real}$:

\begin{prop}[{\cite[Theorem 3.6]{rempe_vanStrien}}]
\label{prop:noboundedWD}
 Let $f\in\B_{\real}$. Let $A_{\bdd}$ be the set of points
  $z\in\C$ whose $\omega$-limit set is a compact subset of
  the real line and which are not contained in attracting or
  parabolic basins. Then $A_{\bdd}$ has empty interior. 

  In particular, if $f\in \B_{\real}^*$, then $f$ has no wandering domain
   whose set of limit functions is bounded. 
\end{prop}

We now state a geometric condition on a function
 $f\in\B_{\real}^*$ that will allow us to
 prove the absence of wandering domains. As we shall see in the
 next section, this condition holds whenever $f$ satisfies the
 hypotheses of Theorem \ref{thm:main}. 

\begin{defn}[Sector condition]
\label{defn:real_sec_cond}
 For $f\in\B_{\real}$, we set
\begin{equation*}
 \Sigma_0(f) := \{\sigma\in\{+,-\}: \lim_{x\to+\infty}
           |f(\sigma x)|=\infty\}.
\end{equation*}
For $\sigma\in \Sigma_0(f)$ we say that 
  $f$ satisfies the \emph{(real) sector condition at $\sigma\infty$} if, 
   for all $R>0$, there are $\vartheta>0$ and $R'>0$ such that
  \[ |f( \sigma x + y )| > R \]
  whenever $x>R'$ and $|y|<\vartheta x$.

 Let us also define 
\begin{align*}
\Sigma(f):=\{\sigma\in\{+,-\}: \exists s\in S(f)\cap\R, n_j\to\infty: 
f^{n_j}(s)\to\sigma\infty\}\subset \Sigma_0(f).
\end{align*}
We say that $f$ satisfies the \emph{(real) sector condition} if 
$f$ satisfies the sector condition at $\sigma\infty$ for all $\sigma\in\Sigma(f)$.
\end{defn}

It is well-known (and easy to see)
 that the sector condition does not change if we
 require it to only hold for \emph{some} $R$, provided $R$ is chosen
 so large that $|s|<R$ for all $s\in S(f)$. 
 (This will also follow from the proof of Theorem \ref{thm:log_der}).

We also note the following standard fact.
\begin{lem} \label{lem:A}
 Let $f\in\B_{\real}$, and let $M>0$ be sufficiently large. Then the set
  \[ A := \bigcup_{\sigma\in \Sigma(f)} \sigma\cdot [M,\infty) \]
  satisfies $A\subset I(f)$ and $f(A)\subset A$. 
\end{lem}
\begin{proof}
  This follows e.g.\ from the Ahlfors distortion theorem, which
   implies that
\[
 \liminf_{x\to +\infty} \frac{\log\log\vert f(\sigma x)\vert}{\log x}\geq \frac{1}{2} \] 
  whenever $\sigma\in \Sigma_0(f)$.  
\end{proof}

\begin{thm}
\label{thm:sectors}
If $f\in\B_{\real}^{*}$ satisfies the real sector condition, 
then $f$ has no wandering domains.
\end{thm}
\begin{proof}
Assume that $f$ has a wandering domain, say $W$. Then the
 set of limit functions on $W$ is contained in the postsingular set,
 and hence in the real axis. 
 Since $I(f)\subset J(f)$ (by the result of Eremenko and Lyubich mentioned
 above), 
 we have $f^n(W)\cap\R\cap I(f)=\emptyset$
 for all $n\in\N$. Furthermore,
 Proposition \ref{prop:noboundedWD} shows that
 the set of limit functions in $W$ cannot be bounded. Thus
 Lemmas \ref{lem:singvalue} and
   \ref{lem:A} show that there is $s\in S(f)$ and a sequence $(n_k)$
  such that $f^{n_k}\to s$ and $f^{n_k-1}\to\infty$. 

 Let us define $U := \C\setminus (A\cup P(f))$, where
  $A$ is as in Lemma \ref{lem:A}. Note that the orbit of $W$ is
  contained in $U$, since $A$ is contained in $I(f)$ and
  every point in $P(f)$ has either an escaping or bounded orbit. 

Now we are in a position to apply Theorem \ref{thm:expansion}. Let
 $D=D_{\eps}(s)\cup U$ and fix $R>|s|+\eps$. The sector condition implies
that there is $R'>0$ and $c\in (0,1)$ such that
 $|f(z)|>R$ whenever $x\geq R'$, $\sigma\in \Sigma(f)$ and
 $|z-\sigma x|\leq c\cdot x$. 

 Set $V := f^{-1}(D)$. 
 Exactly as in the proof of Theorem \ref{thm:mainescaping},
  it follows that 
  \[ \limsup_{n\to\infty} d_U(f^{n-1}(w),U\setminus V) < \infty, \]
 which contradicts Theorem \ref{thm:expansion}.
\end{proof}

\section{Functions that satisfy the real sector condition}

Let us now formulate some conditions to ensure that  
a function satisfies the real sector condition. 
Our first result shows that this property can be
 expressed in terms of an estimate on the 
 logarithmic derivative. 

\begin{thm}
\label{thm:log_der}
 Let $f\in\B_{\real}$ and $\sigma\in \Sigma_0(f)$. 
 Then $f$ satisfies a sector condition at $\sigma\infty$ 
 if and only if there exist constants $r,K>0$ with
\begin{equation}\label{eqn:log_der}
 \frac{\vert f'(\sigma x)\vert}{\vert f(\sigma x)\vert}
\leq K\cdot\frac{\log\vert f(\sigma x)\vert}{x}
\end{equation}
for all $x\geq r$.
\end{thm}
\begin{remark}[Remark 1]
 This concludes the proof of Theorem \ref{thm:main}.
\end{remark}
\begin{remark}[Remark 2]
 We note that the opposite inequality holds for every function 
  $f\in\B$ and a suitable constant $K$; this is an immediate consequence of 
the expansion property of (logarithmic lifts of) maps in the class $\B$ 
(compare e.g. \cite[Lemma 1]{eremenko_lyubich_2}) and also
 follows from (\ref{eqn:log_der_sector}) below.
 If $f\in \B$ has finite order of
 growth, then $f$ satisfies
 (\ref{eqn:log_der}) 
  for all $x$ outside a set of
 finite logarithmic measure; this can also be deduced from (\ref{eqn:log_der_sector}).
\end{remark}
\begin{proof}
 Let $R \geq 1+ \max_{s\in S(f)}|s|$ and set
  $D^* := \C\setminus \overline{D_R(0)}$. 
  Denote by $T$ the component of
  $f^{-1}(D^*)$ that contains
  $\sigma x$ for sufficiently large $x$. 
  ($T$ is
  called a \emph{tract} of $f$.) Since $f:T\to D^*$ is
  a covering map and $f$ is transcendental, $T$ is simply connected. 
 Recall from Proposition \ref{prop:Koebe_hyp_metric}
 that 
\[ \frac{1}{2\dist(z,\partial T)}\leq \rho_T(z)\leq  \frac{2}{\dist(z,\partial T)}.
\] 
 Since $D^*$ is mapped conformally to the punctured unit disc by 
 $z\mapsto R/z$, we have 
\[ \rho_{D^*}(z) = 
    \frac{1}{|z|\log\frac{|z|}{R}}
 \]
 for all $z\in D^*$. In particular,
\begin{align*}
\frac{1}{\vert z\vert\cdot\log\vert z\vert}
\leq\rho_{D^*}(z) 
\leq\frac{2}{\vert z\vert\cdot\log\vert z\vert}
\end{align*}
for $|z|>R^2$. 
Since $f\vert_T$ is a covering map, we have 
$\rho_T(z) = \rho_{D^*}(f(z))\cdot \vert f'(z)\vert$. Combining this
 with the above estimates, we see that
\begin{align}
\label{eqn:log_der_sector}
\frac{1}{4\cdot\dist(z,\partial T)}\leq 
\frac{\vert f'(z)\vert}{\vert f(z)\vert\cdot\log\vert f(z)\vert}
\leq \frac{2}{\dist(z,\partial T)}
\end{align}
when $|z|>R^2$. So we see that
\begin{equation} \label{eqn:sector_dist}
   \dist(\sigma x, \partial T)\geq \eps x 
\end{equation}
holds for some $\eps>0$ and all sufficiently large $x$ if and only if
 (\ref{eqn:log_der}) 
 is satisfied for some $K>0$ and all sufficiently large $x$. 
 Now $f$
 satisfies the 
 sector condition at $\sigma\infty$
 if and only if (\ref{eqn:sector_dist}) holds. The claim follows.
 (Note that, since (\ref{eqn:log_der}) is independent of $R$,
 the sector condition is also independent of $R$, provided
 $R$ is chosen to be of size at least
 $1+\max_{s\in S(f)}|s|$.)
\end{proof}

The real sector condition is preserved under precomposition with appropriate 
polynomials.
\begin{cor}
\label{cor:comp_pol}
 Assume that $f\in\B_{\real}$ satisfies the real sector condition 
  at $\sigma\infty$ for all $\sigma\in\Sigma_0(f)$.
 If $p$ is a real polynomial, then 
  $g := f\circ p$ also satisfies 
  the real sector condition at $\sigma\infty$ for every
  $\sigma\in\Sigma_0(g)$.
\end{cor}
\begin{proof}
 This follows from the fact that the preimage of a sector under
  any polynomial 
  $p$ again contains a sector. Indeed, we have
  $\arg(p(z))\approx d\arg z$, where $d$ is the degree of $p$. 

 Alternatively, it is easy to verify (\ref{eqn:log_der}) for
  $g$ from the corresponding condition for $f$. 
\end{proof}

To give some explicit examples, let us consider a transcendental entire 
function of the form
\begin{align}
\label{eqn:neg_zeroes}
 f(z):=\prod\limits_{n=1}^{\infty} \left( 1+\frac{z}{x_n}\right),\quad 0<x_n\leq x_{n+1}.
\end{align}
So $f$ is a real function of order $\rho(f)<1$ 
with only negative real zeros, but infinitely many of these.

If $f$ is chosen such that 
$\sup\vert f(x)\vert <\infty$
for $x<0$ then $f$ belongs to the class $\B$. Indeed, by definition,
 $f$ is the locally uniform limit of polynomials with only real zeros 
 (i.e., $f$ belongs to the \emph{Laguerre-Polya class}). Since all
 critical points of the approximating polynomials belong to the
 negative real axis, the same is true of $f$; hence the set of
 critical values of $f$ is bounded. On the other hand, the set of
 asymptotic values of $f$ is finite by the Denjoy-Carleman-Ahlfors theorem.

 We have
\begin{align*}
 x\cdot\frac{f'(x)}{f(x)} =x\cdot \sum\limits_{n=1}^{\infty}\frac{1}{x_n + x} 
=  \sum\limits_{n=1}^{\infty} \left( 1 - \frac{x_n}{x_n + x}\right)
=\sum\limits_{n=1}^{\infty} \left( 1 - \frac{1}{w_n(x)}\right),
\end{align*}
where $w_n(x):=(x_n + x)/x_n=1+x/x_n$. 
Note that for a fixed $x$, the sequence $(w_n)=(w_n(x))$ is 
decreasing and converges to $1$. 
Since $\log y>1-1/y$ for $y>1$,  
we obtain that 
 \[ x \cdot \frac{f'(x)}{f(x)}\leq \sum_{n=1}^{\infty} \log w_n =
               \log\prod_{n=1}^{\infty} ( 1 + \frac{x}{x_n}) = 
              \log f(x). \] 
Thus Theorem \ref{thm:log_der} implies that $f$ 
satisfies the real sector condition.
 (Alternatively, this is also easy to check from the definition
  of $f$ and the original statement of the sector condition.) 

Many functions can be written as $f\circ p$, where $f$ is as in (\ref{eqn:neg_zeroes})
and $p$ is a polynomial as in Corollary \ref{cor:comp_pol}, e.g.
\begin{align*}
F_{2k}(z):=\int\limits_{0}^{\infty} \exp(-t^{2k})\cosh(tz) dt,\ k>0,\quad\text{or}\quad 
I_0(z):= 1 + \sum\limits_{k=0}^{\infty} \left(\frac{z^k}{2^k k!}\right)^2.
\end{align*}
These maps occur for $p(z)=z^2$ and $f$ of order $1/2$, converging to $0$ along the
negative real axis \cite[8.4, 8.63]{titchmarsh}. This representation is shared 
 by the function $\sinh z/ z$. Hence
 Corollary \ref{cor:sinh} is a consequence of
 Theorem \ref{thm:main}. (Of course it is also very simple
 to check either (\ref{eqn:log_der}) or the sector condition explicitly
 for these functions.)

Finally, we remark that
 the sector condition can also be phrased by locating the preimages of
 certain compact sets. To do so,  let us define \emph{truncated sectors}
\begin{align*}
 S_{\sigma,\vartheta,R}:=\{\sigma x+iy\in\C:\; x > R, \vert y\vert < \vartheta x\},
\end{align*}
where $\sigma\in\{+,-\}$, $\vartheta>0$ and $R>0$.

\begin{thm}
\label{thm:nonsing_val_omitted}
 Let $f\in\B_{\real}$ and let $\sigma\in\Sigma_0(f)$.
Then the following statements are equivalent:
\begin{itemize}
 \item[(i)] $f$ satisfies the sector condition at
    $\sigma\infty$.
 \item[(ii)] If $K\subset \C$ is compact, then
   $f^{-1}(K)$  
   omits some truncated sector $S_{\sigma,\theta,R}$.
 \item[(iii)] There exists a point $w$ in the unbounded component of 
   $\C\setminus S(f)$ such that
   $f^{-1}(w)$ omits a truncated sector $S_{\sigma,\theta',R'}$.
\end{itemize}
\end{thm}
\begin{proof}
Let us assume that $\sigma=+$ and that
 $f(x)\to+\infty$ as $x\to \infty$; the other cases then follow
 by pre- resp.\ post-composing $f$ with the map $x\mapsto -x$.

Obviously (ii) implies (iii). Also, (i) is clearly equivalent to
 (ii). (The sector condition just states that, for every $R>0$,
 the set $f^{-1}(\overline{D_R(0)})$ omits a truncated sector.) 

So, it remains to show that (iii) implies (i). 
 We show the contrapositive, so assume that there is
 a sequence $(z_k)_{k\geq 0}$, $z_k=x_k + iy_k$, such that 
 $x_k\to+\infty$, 
 $y_k/x_k\to 0$ and $\limsup |f(z_k)|<\infty$. By continuity of $f$,
 we may assume that $|f(z_k)|=R$ is constant and 
 $R>\max_{s\in S(f)}|s|$. (Replace $y_k$ by a smaller value $y_k'$
 if necessary; recall that $f(x_k)\to\infty$ as $k\to\infty$.)

 Let $w\in \C\setminus S(f)$; we will show that
 $f^{-1}(w)$ does not omit any 
 truncated sector; more precisely, we shall show that there
 is a sequence $(w_k)\subset f^{-1}(w)$ with
 $|w_k-x_k|/x_k\to 0$.
 Let $K$ be a full subset of $\C$ that contains $S(f)$ but does 
 not contain $w$ or any point of modulus $R$, and 
 define $U:=\C\setminus K$. (Here \emph{full} means that $K$ is compact and
 that $K$ and the complement of $K$ are both connected.) 
  Let $T$ be the component of
  $f^{-1}(U)$ that contains $x$ for all sufficiently large real $x$.

 Note first that, since $f(x_k)\to\infty$, we have
  $d_U(f(x_k),f(z_k))\to\infty$. By Pick's theorem, $f\vert_T:T\to U$ is
  locally a Poincar\'{e} isometry, hence 
  $d_T(x_k,z_k)\to\infty$ as well. 
  Recall that $|x_k-z_k|=|y_k|$ and $|y_k|/x_k\to 0$. Hence Proposition
  \ref{prop:Koebe_hyp_metric} implies that
   \[ \dist(z_k,\partial T)/x_k \to 0. \] 
  Let $\delta$ be the maximal hyperbolic distance in $U$ between $w$ and
  a point of $\partial D_R(0)$. For every
  $k$, we can then find some $w_k\in f^{-1}(w)\cap T$ such that
  $d_T(z_k,w_k)\leq \delta$. Again using 
   Proposition
  \ref{prop:Koebe_hyp_metric}, we must have
   \[ |w_k -z_k|/x_k \to 0, \]
  and hence $|w_k-x_k|/x_k\to 0$, as claimed. 
\end{proof}

We conclude this section with an observation
concerning the perturbation of a function with the real sector property
within the so-called Eremenko-Lyubich parameter spaces (see
 \cite{eremenko_lyubich_2,boettcher}).
\begin{cor}
Let $f\in\B_{\real}$ and let
$\phi,\psi:\C\to\C$ be quasiconformal homeomorphisms fixing the real line 
such that 
\begin{align*}
 g:=\psi^{-1}\circ f\circ \phi
\end{align*}
is an entire function. Then $f$ satisfies the real sector condition if and only if $g$ does.
\end{cor}
\begin{proof}
 The preimage of a truncated sector around the real axis under the
  quasiconformal map $\phi$ again contains a sector; this follows
  e.g.\ from the geometric characterization of quasicircles. 
\end{proof}

\section{Further results}\label{sec:further}

 \subsection*{An extension of Theorem \ref{thm:expansion}}
   We now return to the general setting of Theorem \ref{thm:expansion}, and
    prove two further statements in a similar spirit. The first strengthens
    (\ref{eqn:zhengestimate}), provided that we choose the domain $W$ to be maximal,
    and is based on 
    similar ideas as \cite[Lemma 3]{bergweilerinvariantdomainssingularities}.
    The second shows that superattracting basins are the only
    domains of normality for which an isolated point of the postsingular set
    can be a constant limit function.
    This is based
    on the same ideas that were applied by Zheng in 
    \cite{zhengahlfors}. 
\begin{thm} \label{thm:zhengbergweiler}
 Let $U$ be a hyperbolic Riemann surface, let $U'\subset U$ be 
  open and let
  $f:U'\to U$ be a holomorphic covering map.

  Let $\W\subset U'$ be the set of all points in $U'$ which have a neighborhood
   $W$ such that $f^n(W)\subset U'$ for all $n$. Then
  \begin{equation}\label{eqn:extendedzheng}
    \lim_{n\to\infty} d_U(f^n(w),U\setminus \W) =\infty 
  \end{equation}
   for all $w\in\W$. In particular, either
   $(f^n(w))$ has no accumulation point in $U$, or the orbit of $w$ 
   eventually enters a cycle of 
   components of $U$ such that
   $f^p:U\to U$ is a conformal isomorphism for some $p\geq 1$.

 Furthermore, let $w\in \W$, and let $W_n$ be the component of
  $\W$ containing $f^n(w)$. Suppose that $z_0$ is a puncture of $U$ and that
  $f^{n_k}|_{W_0}\to z_0$ locally uniformly for some sequence $n_k\to\infty$.
  Then, for sufficiently large $k$, $W_{n_k}\cup \{z_0\}$ is
   simply connected, and $f^p$ extends holomorphically to $z_0$
   with $f^p(z_0)=z_0$ and $(f^p)'(z_0)=0$ for some $p$. In particular,
   $f^{n_k+np}|_W\to z_0$ locally uniformly as $n\to\infty$. 
\end{thm}
\begin{remark}
 By saying that $z_0$ is a puncture of $U$, we mean that
  there is a Riemann surface $X$ such that $U=X\setminus\{z_0\}$. 
\end{remark}
\begin{proof}
  To prove the first statement, we establish the contrapositive. That is, suppose that 
   $W\subset \W$ is open and forward-invariant with $w\in W$, and suppose that 
     \[ d_U(f^{n_k}(w),U\setminus W) \leq C < \infty \]
    for some constant $C$ and some sequence $n_k\to\infty$. We shall show that
   $\partial W$ contains a point of $\W$. Let us set $U_k := f^{-n_k}(U)$.

  Indeed, for each $k$ let $\zeta_k$ be a point of $\partial W$ that is closest to 
    $f^{n_k}(w)$ in the hyperbolic metric of $U$, and hence has hyperbolic distance
   at most $C$ from $f^{n_k}(w)$. Then the hyperbolic geodesic of $U$
   connecting $f^{n_k}(w)$ and $\zeta_k$ belongs to $W$ apart from the
   endpoint $\zeta_k$. Pulling back this curve under $f^{n_k}$, we obtain a curve
   in $W\cap U_k$ that connects $w$ to some point $z_k\in \partial W$. Since
  $f^{n_k}:U_k\to U$ is a covering map, we see that 
   $d_{U_k}(w,z_k)\leq C$. 

  Let $z_0$ be a limit point of the sequence $z_k$ (note that such a limit point
    exists and belongs to $U$, since the points $z_k$ have uniformly bounded hyperbolic
    distance from $w$). Recall that  the points $z_k$ have uniformly bounded hyperbolic
    distance from $w$ in $U_k$, and 
   that
   $d_U(w,\partial U_k) > d_U(w,\partial W)>0$ for all $k$. It follows that
   $d_U(z_k,\partial U_k)$, and thus $d_U(z_0,\partial U_k)$, is
   uniformly bounded from below. Hence $z_0$
    indeed has a neighborhood on which all iterates of $f$ are defined; i.e.\
    $z_0\in \W$ as desired.

 \smallskip

 Next, suppose that the orbit of $f^{n_k}(w)$ has an accumulation point in $U$.
  Let $W_n$ be the component of $\W$ containing $f^{n}(w)$. 
  Clearly, by \eqref{eqn:extendedzheng}, we have
  $W_{n_k}=W_{n_{k+1}}$ for sufficiently large $k$. Hence
  $f^p:W_{n_k}\to W_{n_k}$ is a covering map, where $p=n_{k+1}-n_k$. Since the
  orbit of $w$ accumulates at some point of $W_{n_k}$, it follows that 
  $f^p|_{W_{n_k}}$ is a conformal isomorphism (indeed, this restriction either has
  finite order or is conjugate to an irrational rotation of a disk or an annulus). 

  Now let us prove the final claim, so assume that $X$ is a Riemann surface
   with $U\subset X$, that $z_0$ is an isolated point of $X\setminus U$ and that
   $f^{n_k}(w)\to z_0$. Then we have 
   $f^{n_k}|_{W_0}\to z_0$ locally uniformly by Pick's Theorem.
  Let $D$ be a small
  simply connected neighborhood of $z_0$ with $D\setminus\{z_0\}\subset U$ and set 
  $D^* := D\setminus \{z_0\}$.  

 \begin{claim}
  For sufficiently large $k$, the set $W_{n_k}$ contains a simple closed curve
   $\gamma\subset D^*$ that surrounds $z_0$ (in particular, $W_{n_k}$
   is multiply connected).
 \end{claim}
 \begin{subproof}
  Let $\phi:D\to \D$ be a conformal isomorphism with
    $\phi(z_0)=0$ and set $r_k := |\phi(f^{n_k}(w))|$ (for sufficiently large $k$).
    If we define $\gamma_k := \phi^{-1}(\{z\in \D: |z|=r_k\})$, then 
     the hyperbolic length of $\gamma_k$ in $D^*$, and hence in $U$, tends
     to zero as $k\to \infty$. 

  By (\ref{eqn:zhengestimate}), or \eqref{eqn:extendedzheng},
   it follows that
   $W_{n_k}$ for sufficiently large $k$, proving the 
   claim. 
 \end{subproof}

 By replacing $W_0$ with $W_{n_k}$, for $n_k$ sufficiently large,
   we can assume that $W_0$ contains a curve $\gamma$ as in the
   claim. Note that the component of $U\setminus\gamma$ that contains $\gamma$ on its
   boundary and is not contained in
   $D^*$ cannot be simply connected or conformally equivalent to a punctured disc,
   since every connected component of $U$ is hyperbolic. Let us fix
    $n=n_k$ sufficiently large to ensure that $\gamma$ surrounds $f^n(\gamma)$.

 Let $V^*\subset U$ be 
  the component of $f^{-n}(D^*)$ that contains
  $\gamma$. Since $V^*$ contains $\gamma$, $V^*$ must be multiply connected.
  Since
  $f^{n}:V^*\to D^*$ is a covering map,
  it follows that $V_k^*$ is a punctured disc, and thus 
  $V^*=V\setminus \{z_0\}$, for some simply connected domain $V$. 
  Thus $f^n$ extends holomorphically to $z_0$ with $f^n(z_0)=z_0$. 

 Let $G_0$ be the complementary component of $\gamma$
    that
    contains $z_0$, and set $G_0^* := G_0 \setminus \{z_0\}$. 
    For $m\geq 0$, let $G_m^*$ be the component of $f^{-mn}(G_0^*)$
    that contains $\gamma$. Then, as above, $G_m^*$ is conformally equivalent
    to a punctured disc, and $G_m := G_m^*\cup \{z_0\}$ is simply connected. 
    Let us set
      \[ G := \bigcup_{m\geq 0} G_m\quad\text{and}\quad G^* := G\setminus \{z_0\}. \]
    Then $G$ is simply-connected, and $f^m:G^*\to G^*$ is a covering map. It follows that
    $G^* = W_0$, and that 
   \[ F:G\to G; z\mapsto\begin{cases}
      f^{n}(z) & \text{if } z\in G^* \\
       z_0 &\text{if }z=z_0 \end{cases}\] 
   is an analytic self-map of $G$, and in fact a branched covering map, 
   with the only possible
    branch point at $z_0$.

   Since
    $F(\cl{G_0})\subset G_0$ by choice of $n$ we see that $z_0$ is 
    an attracting or superattracting fixed point. In particular, $F$ is not a conformal isomorphism. 
    Since the only possible branch point of $F$ is at $z_0$, we see that $z_0$ is a critical point; i.e.\
    $z_0$ is a superattracting fixed point of $F$. This completes the proof.
\end{proof}

\subsection*{Applications to Ahlfors islands maps}
 Our results can be applied in a very general framework of one-dimensional 
   holomorphic dynamics, as introduced by Epstein.

 \begin{defn}
  Let $X$ be a compact connected Riemann surface, let $V\subset X$ be open and nonempty,
    and let $f:V\to X$ be holomorphic and nonconstant.

   Then the \emph{Fatou set} of $f$ consists of those points $z\in X$ that have
     a neighborhood $U\subset X$ such that either the iterates 
     $f^n$ are defined on $U$ for all $n\geq 0$ and form a normal
     family there, or $f^n(U)\subset X\setminus\cl{V}$ 
     for some $n\geq 0$. The \emph{Julia set} of $f$ is given by
    $J(f) := X\setminus F(f)$. 

   A component $U$ of $F(f)$ is a \emph{wandering domain} of $f$ if
    $f^n(U)\cap f^m(U)=\emptyset$ for $n\neq m$.

  The set $S(f)$ of \emph{singular values} of $f$ is the smallest compact set
    $S\subset X$ for which $f:f^{-1}(X\setminus S)\to X\setminus S$ is
    a covering map. The \emph{postsingular set} of $f$ is given by
     \[ P(f) := \cl{\bigcup_{n\geq 0} f^n(S(f))}. \]
 \end{defn}
  
 Epstein \cite{adamfinitetype} 
  defined a class of holomorphic functions $f:W\to X$ as above,
  known as \emph{Ahlfors islands maps},  which includes all rational maps, 
  all transcendental entire functions and all 
  meromorphic functions and their iterates (as well as many more), and for which 
  all the basic results of the standard Julia-Fatou theory remain true. We shall not
  state or require the formal definition here, which can be found
  e.g.\ in \cite{hypdim}, along with some background and references.
  For further examples of Ahlfors islands maps, compare \cite{ahlforsmaps}. 

 We begin by generalising the previously mentioned result from \cite{bergweiler_etal} and
  \cite{zhengahlfors} regarding limit functions in a wandering domain.

\begin{prop} \label{prop:bergweileretal}
  Let $f:V\to X$ be holomorphic and nonconstant,  
   where $X$ is a compact connected Riemann surface
   and $V\subset X$ is open and nonempty. Suppose that $F(f)$ has a
    wandering component $W$, and let $w\in W$.

  Then every limit point of the sequence $(f^n(w))$ is a non-isolated point of
    $P(f)$.
\end{prop}
\begin{proof}
  Set $U := X\setminus P(f)$. If $U$ is not hyperbolic, then either $f$ is 
   a conformal automorphism of the Riemann sphere or of a torus, or
   $f$ is a linear endomorphism on a torus, or  
   $X$ is the
   Riemann sphere and $f$ is conformally conjugate to a power map
   $z\mapsto z^d$, $d\in\Z\setminus\{-1,0,1\}$. None of these examples
   have wandering domains, and hence we may assume that $U$ is hyperbolic.

  If we set $U' := f^{-1}(U)$, then $U'\subset U$ and $f:U'\to U$ is
   a covering map. If $f^n(W)\cap P(f)\neq\emptyset$ for some $n\geq 0$, then
   our conclusion holds. (If $f^{n_k}(w)\to a$, then $f^{n_k}|_W\to a$ locally uniformly,
   and hence $a$ is in the accumulation set of an orbit in $P(f)$.)
   Otherwise,  we can apply Theorem
   \ref{thm:zhengbergweiler} to conclude that no limit point of
   $f^n(w)$ can be an isolated point of $X\setminus U = P(f)$, as claimed.
\end{proof}

 We now also apply the first half of Theorem \ref{thm:zhengbergweiler} in
   this general setting. To do so, it is useful to 
   also use natural conformal metrics on non-hyperbolic Riemann surfaces.
   For such a surface $S$, we fix a metric of constant curvature $0$ or $1$
   (the precise normalization in the case of curvature $0$ will not matter
    in our context), and use $d_S$ to denote the distance with respect to this
    metric. 

 \begin{prop} \label{prop:distancetojulia}
  Let $f:V\to X$ be holomorphic and nonconstant, where $X$ is a compact 
   connected Riemann surface
   and $V\subset X$ is open and nonempty.  Also assume that $f$ does not
   extend to a conformal automorphism of $X$. Set $U := X\setminus P(f)$, and let
   $w\in U\cap F(f)$. 

  Then $d_U(f^n(w) , U\cap J(f) )\to \infty$ as $n\to \infty$
    (where we use the convention that the distance is infinite if
     $U\cap J(f)=\emptyset$). 
 \end{prop}
 \begin{proof}
  Again, let us begin by handling the case where $U$ is not hyperbolic. If
    $f$ is an affine toral endomorphism, but not an automorphism of the torus, then
    $F(f)=\emptyset$, and there is nothing to prove. On the other hand, if
    $f$ is conformally conjugate to a power map, then $F(f)$ has exactly two components,
    on which the iterates converge locally uniformly to the two superattracting points in
    $P(f)$, and the claims follow.

  So now suppose that $U$ is hyperbolic. Then the theorem follows directly from
    formula~\eqref{eqn:extendedzheng} in 
   Theorem \ref{thm:zhengbergweiler}. 
 \end{proof}

\subsection*{Baker domains of meromorphic functions} We now apply the
   preceding proposition to prove Theorem \ref{thm:baker}. In fact,
   we can state a more general result, which includes all meromorphic functions
   and their iterates, but also e.g.\ the type of functions studied in
   \cite{zhengahlfors}.
 \begin{prop}
  Let $V\subset\C$ be an unbounded domain, and let 
   $f:V\to\Ch$ be a meromorphic function that does not extend analytically to
   a neighbourhood of 
   $\infty$. Suppose that $W$ is a component of $F(f)$ with
    $f(W)\subset W$ and
   $f^n|_W\to \infty$. Then 
    there exists a sequence $p_n\in P(f)\cap\C$ with 
    \[ |p_n|\to\infty, \quad
       \sup_{n\geq 1} \frac{|p_{n+1}|}{|p_n|} < \infty, \quad\text{and}\quad
      \frac{\dist(p_n,W)}{|p_n|}\to 0. \]
 \end{prop}
 \begin{proof}
   Let us first suppose that $W\cap P(f)=\emptyset$, and define
   $U := \Ch\setminus P(f)\supset W$. Fix some point $w\in W$ and set
      $w_n := f^{n}(w)$; by passing to a forward iterate, we can
      assume that $|w_n|\geq 1$ for all $n$.
      By Proposition \ref{prop:distancetojulia}, we have 
     $d_U(w_n,\partial W) \to \infty$. 

 \begin{claim}
  For every $n$, there is a point 
      $\zeta_n\in \partial W$ with $|w_n|/C_2 < |\zeta_n| < C_2|w_n|$,
    where $C_2$ is a constant independent of $n$.
\end{claim}
\begin{subproof}
    This follows from a result of Bonfert \cite[Theorem 1.2]{bonfertiteration}, 
    which implies that there is a constant $C_1$ (depending on $w$) such that
      \begin{equation}\label{eqn:bonfert} 
       \rho_W(w_n) \geq \frac{1}{C_1\cdot \dist(w_n,\partial W)}. 
      \end{equation}
     Indeed, if $W$ contained annuli of arbitrarily large moduli around some
      $w_n$, then the hyperbolic density of $W$ at $w_n$ tends to zero faster than
      indicated by \eqref{eqn:bonfert}. 

  More precisely, 
       let $C>1$ and consider the annulus
        \[ A(C,n) := \{z\in\C: |w_n|/C < |z| < C |w_n|\}. \]
      Then 
       $\rho_{A(C,n)}(w_n) = \pi/(2|w_n| \log C)$ 
        (see e.g.\ \cite[Section 12.2]{beardonminda}). Setting $C_3 := \min\{|z|:z\in \partial W\}$,
        we have $\dist(w_n,\partial W)\leq |w_n|+C_3$. Hence if 
        $C>1$ is such that $A(C,n)\subset W$, then
     \begin{align*} \log C = \frac{\pi}{2|w_n|\rho_{A(C,n)}(w_n)} &\leq
                      \frac{\pi}{2|w_n|\rho_W(w_n)} \\ &\leq
                      \frac{\pi C_1 \dist(w_n,\partial W)}{2|w_n|} \leq
                      \pi C_1\frac{|w_n| + C_3}{2|w_n|}\leq \frac{\pi C_1(1+C_3)}{2}.\qedhere \end{align*}
  \end{subproof} 
                   
     Let $z_n\in \partial W$ with $|w_n|/C_2 \leq |z_n|\leq C_2|w_n|$, and connect
      $w_n$ and $z_n$ in $\cl{A(C_2,n)}$ by a curve $\gamma_n$ that has
      length $\leq C_4\cdot |w_n|$ (where $C_4=\pi + C_2$).  We may assume without
      loss of generality that $\gamma_n$ contains no other points of $\partial W$. 

Let $p_n$ be a point of $P(f)$ that is closest
      to this curve $\gamma_n$.
    We claim that
       \[ \frac{\dist(p_n,\gamma)}{|w_n|} \to 0 \]
     as $n\to\infty$. Indeed, otherwise 
     Proposition \ref{prop:Koebe_hyp_metric} would imply that 
      the hyperbolic length of the curve $\gamma_n$ in
      $U$ remains bounded, which contradicts the fact that the hyperbolic
      distance between $w_n$ and $z_n$ tends to infinity.
     In particular, there is a constant $C_5$ such that
      \[ |w_n|/C_5 \leq |p_n| \leq C_5|w_n|. \]
 
   To conclude the proof in the case where $P(f)\cap W=\emptyset$, 
    we use a result proved by Zheng \cite[Lemma 4]{zhengboundedfatou} (and 
    independently by Rippon 
    \cite[Theorem 1]{ripponbakerdomains}), which
     states that there is a constant $C_6>1$ such that
      \begin{equation}\label{eqn:zhenggrowth} |w_{n+1}| \leq C_6\cdot |w_n|. \end{equation}
     (This is established using the result of Bonfert mentioned above.) 
     So we have
      \begin{align*}
         |p_{n+1}| &\leq C_5 |w_{n+1}| \leq 
              C_5 C_6 |w_n| \leq C_5^2 C_6 |p_n| \quad\text{and}\\
         \frac{\dist(p_{n},W)}{|p_n|} &\leq
          C_5\frac{\dist(p_n, \gamma)}{|w_n|} \to 0. 
      \end{align*}

   If $P(f)\cap W\neq \emptyset$, then we pick
    $p_0=w_0\in P(f)\cap W$ and set $p_n := w_n = f^n(p_0)$. The result then follows from
     \eqref{eqn:zhenggrowth}.
 \end{proof}
 \begin{remark}
  In the case of entire functions, Bergweiler shows that the sequence
   $p_n$ can even be chosen such that
     $|p_{n+1}|/|p_n|\to 1$ (provided that $W\cap S(f)=\emptyset$). His proof also applies in our
     setting, provided that the domain $W$ is simply connected. It is an interesting question
     whether this stronger result holds for multiply connected domains. 
 \end{remark}

\bibliographystyle{amsalpha}
\bibliography{biblio_wandering}

\end{document}